\RequirePackage{fix-cm}
\documentclass[smallextended]{svjour3}       
\smartqed  
\usepackage{graphicx}
\usepackage{amsmath}
\usepackage{amssymb}
\usepackage{multirow}
\usepackage{color}
\usepackage{longtable}
\usepackage{array}
\usepackage{url}
\usepackage{enumerate}
\usepackage[11pt]{extsizes}

\newcommand{\mb}{\begingroup\setlength\arraycolsep{2pt}\def\arraystretch{1.0}\begin{pmatrix}}
\newcommand{\me}{\end{pmatrix}\endgroup}

\begin{document}

\title{Almost Designs and Their Links with Balanced Incomplete Block Designs
}


\author{Jerod Michel         \and
        Qi Wang 
}


\institute{Jerod Michel \at
              Department of Computer Science and Engineering, Southern University of Science and Technology, Shenzhen 518055, China. \\
              \email{michelj@sustc.edu.cn}           
           \and
           Qi Wang \at
              Department of Computer Science and Engineering, Southern University of Science and Technology, Shenzhen 518055, China.\\
              \email{wangqi@sustc.edu.cn}\newline The authors were supported by the National Science Foundation of China under Grant No. 11601220.
}

\date{Received: date / Accepted: date}

\maketitle

\begin{abstract}
Almost designs ($t$-adesigns) were proposed and discussed by Ding as a certain generalization of combinatorial designs related to almost difference sets. Unlike $t$-designs, it is not clear whether $t$-adesigns need also be $(t-1)$-designs or $(t-1)$-adesigns. In this paper we discuss a particular class of 3-adesigns, i.e., 3-adesigns coming from certain strongly regular graphs and tournaments, and find that these are also $2$-designs. We construct several classes of these, and discuss some of the restrictions on the parameters of such a class. We also construct several new classes of 2-adesigns, and discuss some of their properties as well.
\keywords{Almost difference set, difference set, strongly regular graph, $t$-adesign, tournament.}
\subclass{05B05 \and 05E30}
\end{abstract}

\section{Introduction}\label{sec1}
Combinatorial designs are an interesting subject of combinatorics closely related to finite geometry \cite{BET}, \cite{DEM}, \cite{HIR}, with applications in experiment design \cite{FISH}, coding theory \cite{ASS}, \cite{CUN}, \cite{HP} and cryptography \cite{CDR} \cite{STIN}.
\subsection{Finite incidence structures}
A (finite) {\it incidence structure} is a triple $(V,\mathcal{B},I)$ such that $V$ is a finite set of elements called {\em points}, $\mathcal{B}$ is a finite set of elements called {\em blocks}, and $I$ ($\subseteq V\times\mathcal{B}$) is a symmetric binary relation between $V$ and $\mathcal{B}$. Since, in the following, all incidence structures $(V,\mathcal{B},I)$ are such that $\mathcal{B}$ is a collection (i.e., a multiset) of nonempty subsets of $V$, and $I$ is given by membership (i.e., a point $p\in V$ and a block $B\in\mathcal{B}$ are incident if and only if $p\in B$), we will denote the incidence structure $(V,\mathcal{B},I)$ simply by $(V,\mathcal{B})$. An incidence structure that has no repeated blocks is called {\em simple}. All of the incidence structures discussed in the following are assumed to be simple. A $t$-$(v,k,\lambda)$ {\it design} (or $t$-design, for short) (with $0<t<k<v$) is an incidence structure $(V,\mathcal{B})$ where $V$ is a set of $v$ points and $\mathcal{B}$ is a collection of $k$-subsets of $V$ such that any $t$-subset of $V$ is contained in exactly $\lambda$ blocks \cite{BET}. In the literature, $t$-designs with $t=1$ are sometimes referred to as {\it tactical configurations}, and those with $t=2$ are sometimes referred to as {\it balanced incomplete block designs}. We will denote the number of blocks of an incidence structure by $b$, and the number of blocks containing a given subset $A\subseteq V$ of points by $r_{A}^{\mathcal{B}}$ (when $A$ is a singleton, and $(V,\mathcal{B})$ is a tactical configuration, simply by $r^{\mathcal{B}}$). Then the identities \[
bk=vr^{\mathcal{B}},\] and \[r^{\mathcal{B}}(k-1)=(v-1)\lambda\] restrict the possible sets of parameters of $2$-designs. A $t$-design in which $b=v$ and $r^{\mathcal{B}}=k$ is called {\it symmetric}. The {\it dual} $(V, \mathcal{B})^{\perp}$ of the incidence structure $(V,\mathcal{B})$ is the incidence structure $(\mathcal{B},V)$ with the roles of points and blocks interchanged. A symmetric incidence structure has the same parameters as its dual. For an ambient set $V$, and a subset $A\subseteq V$, we will sometimes denote by $\overline{A}$ its complement $V\setminus A$. The $v\times b$ $0$-$1$ matrix whose rows and columns are indexed by $V$ and $\mathcal{B}$, respectively, is called the {\it incidence matrix} of $(V,\mathcal{B})$.
\par For a matrix $M$, let $(M)_{ij}$ denote the $(i,j)$-th entry of $M$. Let $\mathcal{V}_{M}$ denote the set by which the columns of $M$ are indexed, and let $\mathcal{B}_{M}$ denote the set of supports of the rows of $M$. We will denote by $J$ and $I$ the all-one matrix and the identity matrix, respectively (dimensions will be clear from the context).
\subsection{Difference sets and almost difference sets}
One important way of obtaining (symmetric) balanced incomplete block designs is by constructing difference sets \cite{BET}, \cite{STIN}. Let $G$ be a group (written additively) of order $v$, and let $k$ and $\lambda$ be integers satisfying $2\leq k<v$. A $(v,k,\lambda)$ {\it difference set} in $G$ is a $k$-subset $D\subseteq G$ such that the multiset $\{* \ x-y\mid x,y\in D,x\neq y \ *\}$ contains every nonidentity member of $G$ exactly $\lambda$ times. It is not difficult to see that the incidence structure given by $(G,Dev(D))$ is a $2$-$(v,k,\lambda)$ design, where $Dev(D)$, called the {\it development} of $D$, denotes the set $\{D+g\mid g\in G\}$ (where $D+g:=\{d+g\mid d\in D\}$) of translates of $D$ over $G$.
\par
Almost difference sets are a generalization of difference sets. In the literature there are two different definitions of almost difference sets \cite{DAV}, \cite{DING00}. The following unification was given in \cite{DHM}. A $(v,k,\lambda,s)$ {\it almost difference set} in $G$ is a $k$-subset $D\subseteq G$ such that the multiset $\{* \ x-y\mid x,y\in D,x\neq y \ *\}$ contains $s$ nonidentity members of $G$ with multiplicity $\lambda$, and $v-1-s$ nonidentity members with multiplicity $\lambda+1$.
\par
A difference set can be viewed as an almost difference set with $s=0$ or $s=v-1$. The complement $G\setminus D$ of a $(v,k,\lambda,s)$ almost difference set is an almost difference set with parameters $(v,v-k,v-2k+\lambda,s)$. A simple restriction which can be applied to the parameters of almost difference sets is that $(v-1)(\lambda+1)-s=k(k-1)$ must hold for any $(v,k,\lambda,s)$ almost difference set.
\par
Difference sets and almost difference sets also have extensive applications in various fields such as communications, sequence design, error correcting codes, and CDMA and cryptography \cite{CDR}, \cite{CUN}, \cite{GOL}. For a good survey on almost difference sets, the reader is referred to \cite{CUN}.
\subsection{Strongly Regular Graphs and Tournaments}\label{sec1.3}
We will assume some familiarity with graph theory. A graph $\Gamma=(V,E)$ consists of a vertex set $V$ with $\left|V\right|=n$, an edge set $E$, and a relation that associates with each edge a pair of vertices. A {\it strongly regular graph} with parameters $(n,k,\lambda,\mu)$ is a graph $\Gamma$ with $n$ vertices in which the number of common neighbors of $x$ and $y$ is $k,\lambda$ or $\mu$ according as $x$ and $y$ are equal, adjacent or non-adjacent, respectively. The complement of an $(n,k,\lambda,\mu)$ strongly regular graph is an $(n,n-k-1,n-2k+\mu-2,n-2k+\lambda)$ strongly regular graph. For a good introduction to strongly regular graphs the reader is referred to \cite{VAN}. Strongly regular graphs whose parameters are (up to complementation) of the form $(n,\frac{n-1}{2},\frac{n-5}{4},\frac{n-1}{4})$, are called {\it Paley type}, and are closely related to conference matrices \cite{STIN}. A {\it conference matrix} of order $n$ is an $n\times n$ matrix $C$ with diagonal entries $0$, and off-diagonal entries $\pm 1$, which satisfies $CC^{T} =(n-1)I$. Assume $C$ is such a matrix, and let $S$ be the matrix obtained from $C$ by deleting the first row and column, and let $A$ be the matrix obtained from $S$ by replacing $-1$ by $1$, and $1$ by $0$. If $n \equiv 2 \ ({\rm mod} \ 4)$ then $S$ is symmetric, and $A$ is the adjacency matrix of a strongly regular graph of Paley type. Conference matrices of order $n \equiv 2 \ ({\rm mod} \ 4)$ are in fact equivalent to Paley type strongly regular graphs (see \cite{GOETH} and \cite{REID}). Cayley graphs which are strongly regular are equivalent to partial difference sets. For a formulation of strongly regular graphs in terms of partial difference sets, the reader is referred to \cite{VAN}.
\par
A directed graph $\Gamma=(V,\mathcal{E})$ consists of a vertex set $V$ with $\left|V\right|=n$, and a set $\mathcal{E}$ of ordered pairs of vertices (or arcs). A {\it tournament} is a directed graph $\Gamma$ with $n$ vertices in which each pair of vertices $x,y\in V$ are joined by exactly one member of $\mathcal{E}$. The in-degree and out-degree of a vertex $x\in V$ are defined to be the number of arcs of the form $yx$, and the number of arcs of the form $xy$, for $y\in V$, respectively. A {\it doubly regular tournament} is a tournament $\Gamma$ on $n$ vertices in which every vertex has in-degree and out-degree $\frac{n-1}{2}$, and each pair of vertices has $\frac{n-3}{4}$ common out-neighbors, and the same number of common in-neighbors. Doubly regular tournaments are closely related to skew conference matrices. Let $C$ be a conference matrix of order $n$, let $S$ be the matrix obtained from $C$ by deleting the first row and column, and let $A$ be the matrix obtained from $S$ by replacing $-1$ by $1$, and $1$ by $0$. If $n \equiv 0 \ ({\rm mod} \ 4)$ then $S$ is skew-symmetric, and $A$ is the adjacency matrix of a doubly regular tournament. Conference matrices of order $n \equiv 0 \ ({\rm mod} \ 4)$ are in fact equivalent to doubly regular tournaments (see \cite{GOETH} and \cite{REID}). Cayley graphs which are doubly regular tournaments are equivalent to skew Hadamard difference sets. For a formulation of doubly regular tournaments in terms of skew Hadamard difference sets, the reader is referred to \cite{STIN} and \cite{DPW}.
\subsection{Adesigns}
Recent interest in almost difference sets and their codes is the main motivation for studying adesigns (almost designs). Let $V$ be a $v$-set and $\mathcal{B}$ a collection of subsets of $V$, called blocks, each having cardinality $k$. If there is a positive integer $\lambda$ such that every $t$-subset of $V$ is incident with either $\lambda$ blocks or with $\lambda+1$ blocks, and $(V,\mathcal{B})$ is not a $t$-design, then $(V,\mathcal{B})$ is called a $t$-$(v,k,\lambda)$ {\it adesign} (or $t$-adesign for short). It is easy to see that a $0$-$1$ matrix $A$ is the incidence matrix of a $2$-$(v,k,\lambda)$ design $(V,\mathcal{B})$ with repetition number $r=r^{\mathcal{B}}$ if and only if \begin{equation}\label{eqAA}
AA^{T}=rI+\lambda(J-I) \text{ and } A^{T}J=kJ,\end{equation} and is the incidence matrix of a $2$-$(v,k,\lambda)$ adesign $(V,\mathcal{B})$ with constant repetition number $r=r^{\mathcal{B}}$ if and only if there exists a $v\times v$ $0$-$1$ matrix $S$, whose diagonal entries are all zero, such that \begin{equation}\label{eqAB}
AA^{T}=rI+\lambda S + (\lambda+1)(J-I-S) \text{ and } A^{T}J=kJ.\end{equation} The following lemma illustrates the relation between almost difference sets and adesigns. The relation is analogous to that between difference sets and $2$-designs. The proof is easy and so is omitted.
\begin{lemma}\label{le1} Let $D$ be a $(v,k,\lambda,s)$ almost difference set in an abelian group $G$. Then $(G,Dev(D))$ is a $2$-$(v,k,\lambda)$ adesign. Moreover, we have \[
r_{\{x,y\}}^{Dev(D)}=\begin{cases}
\lambda,& \text{ if } \ |(D+x)\cap(D+y)|=\lambda,\\
\lambda+1,&\text{ otherwise,}\end{cases}\]for all distinct $x,y \in G$.
\end{lemma}
\par
Adesigns were first coined by Ding in \cite{CUN}, and several constructions of adesigns and their applications were further investigated in \cite{MDI} and, indirectly in \cite{DYI} and \cite{WW}, as it was shown in \cite{MDI} that almost difference families give $2$-adesigns. It should also be noted that adesigns need not always come from the developments of difference sets or almost difference sets, e.g., there are the duals of quasi-symmetric designs whose block intersection numbers have a difference of one (see Example 5.4 of \cite{MICH00}), as well as those discussed in Example 6.5 of \cite{MDI}. Partial geometric designs, an important type of incidence structure (see \cite{NEUM}) in which one of the defining properties of partial geometries is generalized, were considered in \cite{MICH00}, where an investigation was made into exactly when a $2$-adesign is partial geometric. It was found that, for this to occur, some strong conditions must be satisfied (see Examples 5.4 and 5.5 of \cite{MICH00}). In this paper we will study a special class of $3$-adesigns, i.e., 3-adesigns coming from certain strongly regular graphs and tournaments, and find that these are also $2$-designs. We give several constructions of such $3$-adesigns and we discuss some restrictions on their parameters as well as their links to some other combinatorial objects such as $\lambda$-coverings. Moreover, we construct several new families of $2$-adesigns and discuss some of the restrictions on their parameters.
\par
The remainder of this paper is organized as follows. In Section 2 we make an initial investigation into when a $(t+1)$-adesign is a $t$-design or a $t$-adesign. In Section 3 we give two generic constructions of $3$-adesigns which are balanced incomplete block designs and, furthermore, we discuss the question of when a $3$-adesign is a $2$-design or $2$-adesign. In Section 4 we give some new constructions of $2$-adesigns and we discuss some of the restrictions on their parameters as well. Section 5 concludes the paper with some open problems.
\section{A note on the parameters of $(t+1)$-adesigns which are either $t$-designs or $t$-adesigns}\label{sec2}
It is well-known that $(t+1)$-designs are always $t$-designs (see \cite{BET}). However, it is not clear whether a $(t+1)$-adesign need always be a $t$-design or $t$-adesign. In this section we make a preliminary investigation into when a $(t+1)$-adesign is a $t$-design, or a $t$-adesign, by eliminating some of the possible parameters.
\par
Suppose that $(V,\mathcal{B})$ is a $(t+1)$-$(v,k,\lambda)$ adesign with $b$ blocks. Let $r_{Y}$ denote the number of blocks containing the $t$-subset $Y$ of $V$, and define \[
I_{Y}=\{(z,B)\mid z\in V\setminus Y\text{ and }Y\cup\{z\}\subseteq B\in\mathcal{B}\}.\]We will count $|I_{Y}|$ in two ways. There are $v-t$ ways to choose $z$, and since $(V,\mathcal{B})$ is a $(t+1)$-adesign, neither $|I_{Y}|=\lambda(v-t)$ nor $|I_{Y}|=(\lambda+1)(v-t)$ can hold for all $t$-subsets $Y$ contained in $V$, otherwise $(V,\mathcal{B})$ would be a $(t+1)$-design. Thus $\lambda(v-t)\leq|I_{Y}|\leq(\lambda+1)(v-t)$. We also have $r_{Y}$ ways to choose a block $B$ containing $Y$, and for each choice of $B$, there are $k-t$ ways to choose $z$. This gives us $\lambda(v-t)\leq r_{Y}(k-t)\leq(\lambda+1)(v-t)$ for all possible $t$-subsets $Y$ contained in $V$, whence \begin{equation}\label{eq30}
\lambda\leq r_{Y}\frac{k-t}{v-t}\leq\lambda+1.
\end{equation} Notice that if $\frac{v-t}{k-t}<2$, then \[\left\lceil\lambda\frac{v-t}{k-t}\right\rceil\leq r_{Y}\leq\left\lfloor\lambda\frac{v-t}{k-t}+\frac{v-t}{k-t}\right\rfloor<\left\lceil\lambda\frac{v-t}{k-t}\right\rceil+2,\]so that, as $Y$ runs over the $t$-subsets of $V$, the only possible values for $r_{Y}$ are $\lceil\lambda\frac{v-t}{k-t}\rceil$ or $\lceil\lambda\frac{v-t}{k-t}\rceil+1$. Thus, $(V,\mathcal{B})$ is either a $t$-$(v,k,\lambda')$ adesign with $\lambda'=\lceil\lambda\frac{v-t}{k-t}\rceil$, or a $t$-$(v,k,\lambda')$ design with $\lambda'=\lceil\lambda\frac{v-t}{k-t}\rceil$ or $\lceil\lambda\frac{v-t}{k-t}\rceil+1$. Also notice that if $(V,\mathcal{B})$ is in fact a $t$-design, then by (\ref{eq30}) we must have $\lambda'\frac{k-t}{v-t}-1<\lambda<\lambda'\frac{k-t}{v-t}$ so that $\lambda=\lfloor\lambda'\frac{k-t}{v-t}\rfloor$. Moreover, multiplying through (\ref{eq30}) by $\binom{v}{t+1}/\binom{k}{t+1}$, and taking into account that the inequality must be strict, we have $\lambda\binom{v}{t+1}/\binom{k}{t+1}<\lambda'\binom{v}{t}/\binom{k}{t}<(\lambda+1)\binom{v}{t+1}/\binom{k}{t+1}$ from which it follows that \begin{equation}\label{eq3.1}\lambda\binom{v}{t+1}/\binom{k}{t+1}<b<(\lambda+1)\binom{v}{t+1}/\binom{k}{t+1}.\end{equation}
We have thus shown the following.
\begin{lemma}\label{le30.0} Let $(V,\mathcal{B})$ be a $(t+1)$-$(v,k,\lambda)$ adesign with $b$ blocks. Then for any $t$-subset $Y$ of $V$ we have $\lambda\leq r_{Y}\frac{k-t}{v-t}\leq\lambda+1$. If $\frac{k-t}{v-t}>\frac{1}{2}$ then $(V,\mathcal{B})$ is either a $t$-$(v,k,\lambda')$ adesign with $\lambda'=\lceil\lambda\frac{v-t}{k-t}\rceil$, or a $t$-$(v,k,\lambda')$ design with $\lambda'=\lceil\lambda\frac{v-t}{k-t}\rceil$ or $\lceil\lambda\frac{v-t}{k-t}\rceil+1$. Moreover, if $(V,\mathcal{B})$ is a $t$-$(v,k,\lambda')$ design, then $\lambda=\lfloor\lambda'\frac{k-t}{v-t}\rfloor$ and $\lambda\binom{v}{t+1}/\binom{k}{t+1}<b<(\lambda+1)\binom{v}{t+1}/\binom{k}{t+1}$.
\end{lemma}
It should be noted that the purpose behind stating the inequality in (\ref{eq3.1}) is not to give an estimate on the number of blocks in the $2$-design (since we already know the number of blocks in a $2$-design), but to show the relationship between the number of blocks and $\lambda$, and the resulting strictness of inequality. It should also be noted that, in what follows, it will be made evident that the assumption that $\frac{k-t}{v-t}>\frac{1}{2}$ is sufficient for the $(t+1)$-adesign to be either a $t$-adesign or a $t$-design, but not necessary. The necessary conditions seem difficult to discern, and are left as an open problem. We now give some constructions of $3$-adesigns which are balanced incomplete block designs.
\section{Constructions of $3$-adesigns which are balanced incomplete block designs}\label{sec3}
Both of the constructions in this section follow the same basic method. We take the sets of supports of the rows of adjacency matrices of certain graphs and consider their unions. Note that these are not the first constructions of $3$-adesigns (see Section 6 of \cite{MDI}). However, the constructions in this work are more general, and show definite links between other combinatorial objects such as balanced incomplete block designs and strongly regular graphs.
\par
We first give a construction based on strongly regular graphs. We will need the following lemma, which simply formalizes some of our discussion in Section \ref{sec1.3}.
\begin{lemma}\label{le00} Let $A$ be the adjacency matrix of a strongly regular graph with parameters $(v,k,\lambda,\mu)$. Then \[
A^{2}=kI+\lambda A+\mu (J-I-A) \text{  and  } AJ=kJ.\]
\end{lemma}
Here we give our first construction.
\begin{theorem}\label{th5} Let $A$ be the adjacency matrix of a Paley type strongly regular graph on $n$ vertices, and denote by $A'$ the matrix $J-I-A$. Then $(\mathcal{V}_{A},\mathcal{B}_{A}\cup\mathcal{B}_{A'})$ is a $2$-$(n,\frac{n-1}{2},\frac{n-3}{2})$ design and a $3$-$(n,\frac{n-1}{2},\frac{n-9}{4})$ adesign.
\end{theorem}
\begin{proof} For simplicity, let $\mathcal{V}=\mathcal{V}_{A}$, $\mathcal{B}=\mathcal{B}_{A}$ and $\mathcal{B}'=\mathcal{B}_{A'}$, and denote $\frac{n-1}{2}$ by $k$ and $\frac{n-5}{4}$ by $\lambda$. By Lemma \ref{le00} (and using the fact that $A$ is symmetric) it is easy to see that \[
\mb A & A' \me \mb A \\ A' \me = A^{2} + (A')^{2} = (n-1)I+(k-1)(J-I). \]
Then, by (\ref{eqAA}), and the regularity of the graphs, we have that $(\mathcal{V},\mathcal{B}\cup\mathcal{B}')$ is a $2$-$(n,k,k-1)$ design.
\par
Now we want to show that $(\mathcal{V},\mathcal{B}\cup\mathcal{B}')$ is a $3$-adesign. To count the number of blocks of $\mathcal{B}\cup\mathcal{B}'$ in which $x,y$ and $z$ appear together, we first count the number of blocks of $\mathcal{B}\cup\mathcal{\overline{B}}$ in which $x,y$ and $z$ appear together, where $\mathcal{\overline{B}}=\mathcal{B}_{J-A}$. Let $x,y$ and $z$ be distinct members of $\mathcal{V}$. Suppose that $x,y$ and $z$ appear together in $\omega$ blocks of $\mathcal{B}$.
\newline
{\bf Case 1:} Assume that $(A)_{xy}=(A)_{xz}=(A)_{yz}=1$. (In other words, assume that any two of $x,y$ and $z$ appear together in $\lambda$ blocks of $\mathcal{B}$.) Lemma \ref{le00} together with the principle of inclusion and exclusion implies that there are $n-3k+3\lambda-\omega$ blocks in $\mathcal{\overline{B}}$ containing $x,y$ and $z$. Thus, there are $n-3k+3\lambda$ blocks in $\mathcal{B}\cup\mathcal{\overline{B}}$ containing $x,y$ and $z$. We want to know how many of these correspond to rows of $J-A$ whose indices are $x,y$ or $z$. But if any one of the three blocks corresponding to the rows of $J-A$ indexed by $x,y$ and $z$ contains each of the points $x,y$ and $z$, then two of $(A)_{xy},(A)_{xz}$ and $(A)_{yz}$ must be equal to zero, a contradiction. Thus, in this case, $x,y$ and $z$ appear together in exactly $v-3k+3\lambda=\lambda-1$ blocks of $\mathcal{B}\cup\mathcal{B}'$.
\newline
{\bf Case 2:} Assume that $(A)_{xy}=(A)_{xz}=1$ and $(A)_{yz}=0$. By Lemma \ref{le00}, and the principle of inclusion and exclusion, there are $n-3k+3\lambda+1-\omega$ blocks in $\mathcal{\overline{B}}$ containing $x,y$ and $z$, whence $n-3k+\lambda+1$ blocks in $\mathcal{B}\cup\mathcal{\overline{B}}$ containing $x,y$ and $z$. As in the last case, if we suppose that any one of the three blocks corresponding to the rows of $J-A$ indexed by $x,y$ and $z$ contains each of the points $x,y$ and $z$, then two of $(A)_{xy},(A)_{xz}$ and $(A)_{yz}$ must be equal to zero, again leading to a contradiction. Thus, in this case, $x,y$ and $z$ appear together in exactly $n-3k+3\lambda+1=\lambda$ blocks of $\mathcal{B}\cup\mathcal{B}'$.
\newline
{\bf Case 3:} Assume that $(A)_{xy}=1$ and $(A)_{xz}=(A)_{yz}=0$. By Lemma \ref{le00}, and the principle of inclusion and exclusion, there are $n-3k+3\lambda+2-\omega$ blocks in $\mathcal{\overline{B}}$ containing $x,y$ and $z$, whence $n-1-3k+\lambda+2$ blocks in $\mathcal{B}\cup\mathcal{\overline{B}}$ containing $x,y$ and $z$. The only one of the three blocks corresponding to the rows of $J-A$ indexed by $x,y$ and $z$ that contains each of the points $x,y$ and $z$ is that corresponding to $z$ (i.e., the support of the $z$-th row). Thus, in this case, $x,y$ and $z$ appear together in exactly $n-3k+3\lambda+1=\lambda$ blocks of $\mathcal{B}\cup\mathcal{B}'$. Thus $(\mathcal{V},\mathcal{B}\cup\mathcal{B}')$ is a $3$-$(n,k,\lambda-1)$ adesign.
\end{proof}
There is also the following construction, where the union of the complementary block sets is considered. Its proof is similar and so is omitted.
\begin{corollary} Let $A$ be the adjacency matrix of a Paley type strongly regular graph on $n$ vertices, and denote by $A'$ the matrix $J-I-A$. Then $(\mathcal{V}_{A},\mathcal{B}_{A+I}\cup\mathcal{B}_{A'+I})$ is a $2$-$(n,\frac{n+1}{2},\frac{n-1}{2})$ design and a $3$-$(n,\frac{n-1}{2},\frac{n-1}{4})$ adesign.
\end{corollary}
\begin{example}\label{ex00} If $C$ is a conference matrix of order $n \equiv 2 \ ({\rm mod} \ 4)$, then let $S$ be the matrix obtained from $C$ by deleting the first row and column, and let $A$ be the matrix obtained from $S$ by replacing $-1$ by $1$, and $1$ by $0$. Then $A$ is the adjacency matrix of a strongly regular graph of Paley type (see Section \ref{sec1.3}) and, by Theorem \ref{th5}, $(\mathcal{V}_{A},\mathcal{B}_{A}\cup\mathcal{B}_{A'})$ is a $2$-$(n-1,\frac{n-2}{2},\frac{n-4}{2})$ design and a $3$-$(n-1,\frac{n-2}{2},\frac{n-10}{2})$ adesign.
\end{example}
We now consider a construction based on tournaments. We will need the following lemma, which is also a mere formalization of part of the discussion in Section \ref{sec1.3}.
\begin{lemma}\label{le01} {\rm \cite{GOETH}} Let $A$ be the adjacency matrix of a tournament $\Gamma$ on $n$ vertices, and denote by $S$ the matrix $2A+I-J$. Then $\Gamma$ is doubly regular if and only if \[SS^{T}=nI-J.\]
\end{lemma}
Here we give the construction.
\begin{theorem}\label{th6} Let $A$ be the adjacency matrix of a doubly regular tournament on $n$ vertices, and denote by $A'$ the matrix $J-I-A$. Then $(\mathcal{V}_{A},\mathcal{B}_{A}\cup\mathcal{B}_{A'})$ is a $2$-$(n,\frac{n-1}{2},\frac{n-3}{2})$ design and a $3$-$(n,\frac{n-1}{2},\frac{n-7}{4})$ adesign.
\end{theorem}
\begin{proof} Let $\Gamma=(V,\mathcal{E})$ be the tournament with adjacency matrix $A$. Again, for simplicity, let $\mathcal{V}=\mathcal{V}_{A}$, $\mathcal{B}=\mathcal{B}_{A}$, $\mathcal{B}'=\mathcal{B}_{A'}$, and $\overline{\mathcal{B}}=\mathcal{B}_{J-A}$, and denote $\frac{n-1}{2}$ by $k$ and $\frac{n-3}{4}$ by $\lambda$. The fact that $r_{\{x,y\}}^{\mathcal{B}'}$ is the constant $\lambda$ for all $x,y\in \mathcal{V}$ implies that $(\mathcal{V},\mathcal{B}\cup\mathcal{B}')$ is a $2$-$(n,k,2\lambda)$ design.
\par
We need to show that $(\mathcal{V},\mathcal{B}\cup\mathcal{B}')$ is a $3$-adesign. Like in the previous construction, we will assume that $x,y$ and $z$ appear together in $\omega$ blocks of $\mathcal{B}$, and we will first count the number of blocks of $\mathcal{B}\cup\mathcal{\overline{B}}$ in which $x,y$ and $z$ appear together. By the principle of inclusion and exclusion, there are $n-3k+3\lambda-\omega$ blocks in $\mathcal{\overline{B}}$ containing $x,y$ and $z$. Then there are $n-3k+3\lambda$ blocks in $\mathcal{B}\cup\mathcal{\overline{B}}$ containing $x,y$ and $z$. We want to know how many of these correspond to the rows of $J-A$ whose indices are $x,y$ or $z$. Notice if we suppose that the two blocks corresponding to the rows of $J-A$ indexed by $x$ and $y$ both contain each of the points $x,y$ and $z$, then $\mathcal{E}$ must contain both of the arcs $xy$ and $yx$, a contradiction to Lemma \ref{le01}. Thus, no more than one of the three blocks corresponding to the rows of $J-A$ indexed by $x,y$ and $z$ can contain each of the points $x,y$ and $z$. We need only show that $(\mathcal{V},\mathcal{B}\cup\mathcal{B}')$ is not a $3$-design, and we will be done. If $(\mathcal{V},\mathcal{B}\cup\mathcal{B}')$ were a $3$-design, then, by the above arguments, the only choices for the constant $r_{\{x,y,z\}}^{\mathcal{B}\cup\mathcal{B}'}(=:\lambda')$ would be $\lambda$ or $\lambda-1$. The number of blocks in $\mathcal{B}\cup\mathcal{B}'$ is given by $\lambda'\binom{n}{3}/\binom{k}{3}$, whence the equation \begin{equation}\label{eq8} 2n=\lambda'\binom{n}{3}/\binom{k}{3}\end{equation} must hold. If $\lambda'=\lambda$ then (\ref{eq8}) becomes $n=n+3$, a contradiction, and if $\lambda'=\lambda-1$ then (\ref{eq8}) becomes $(k-1)(k-2)=(\lambda-1)(n-2)$, which again leads to a contradiction.
\end{proof}
In the following, the union of the complementary block sets is considered. The proof is similar and so is omitted.
\begin{corollary} Let $A$ be the adjacency matrix of a doubly regular tournament on $n$ vertices, and denote by $A'$ the matrix $J-I-A$. Then $(\mathcal{V}_{A},\mathcal{B}_{A+I}\cup\mathcal{B}_{A'+I})$ is a $2$-$(n,\frac{n+1}{2},\frac{n+1}{2})$ design and a $3$-$(n,\frac{n+1}{2},\frac{n-3}{4})$ adesign.
\end{corollary}
\begin{example}\label{ex01} If $C$ is a conference matrix of order $n \equiv 0 \ ({\rm mod} \ 4)$, then let $S$ be the matrix obtained from $C$ by deleting the first row and column, and let $A$ be the matrix obtained from $S$ by replacing $-1$ by $1$, and $1$ by $0$. Then $A$ is the adjacency matrix of a doubly regular tournament (see Section \ref{sec1.3}) and, by Theorem \ref{th6}, $(\mathcal{V}_{A},\mathcal{B}_{A}\cup\mathcal{B}_{A'})$ is a $2$-$(n-1,\frac{n-2}{2},\frac{n-4}{2})$ design and a $3$-$(n-1,\frac{n-2}{2},\frac{n-8}{4})$ adesign.
\end{example}
It is clear that the balanced incomplete block designs resulting from the $3$-adesigns constructed from partial difference sets in Example \ref{ex00}, and from skew Hadamard difference sets in Example \ref{ex01}, can also be realized as difference families, each consisting of two distinct difference sets (or almost difference sets). To the best of our knowledge, the only instances of these that have been reported on previously are those constructed via cyclotomy, discussed by Wilson in \cite{WIL} as difference families, and also by Liu and Ding in \cite{LIU} as balanced incomplete block designs.
\par
Assuming that $(V,\mathcal{B})$ is a $2$-$(v,k,\lambda')$ design, the condition $\lambda=\lfloor\lambda'\frac{k-2}{v-2}\rfloor$ stated in Lemma \ref{le30.0} is necessary for $(V,\mathcal{B})$ to be a $3$-$(v,k,\lambda)$ adesign, but not sufficient, as the following example, whose construction method was discussed in {\rm \cite{MDI}}, illustrates.
\begin{example} Let $n$ be an odd integer divisible by $3$. Consider, for fixed $a\in\mathbb{Z}_{n}$, all pairs $\{a-i \ ({\rm mod} \ n),a+i \ ({\rm mod} \ n)\}$, for $i=1,...,\frac{n-1}{2}$. The union of any two distinct pairs gives a block consisting of four points. Denote, for fixed $a\in\mathbb{Z}_{n}$, the set of all blocks obtained in this way by $\mathcal{B}_{a}$. Then $(\mathbb{Z}_{n},\cup_{a\in\mathbb{Z}_{n}}\mathcal{B}_{a})$ is a $2$-$(n,4,n)$ design. Also notice that $\lfloor n\frac{2}{n-2}\rfloor=2$ for all $n\geq 9$, and the number of blocks is $b=n\binom{(n-1)/2}{2}$ so that $2\binom{n}{3}/\binom{4}{3}<b<3\binom{n}{3}/\binom{4}{3}$ is satisfied; however, since $n$ is divisible by 3, we can find $3$-subsets of $\mathbb{Z}_{n}$ not contained in any block (choose three points $x,y$ and $z$ so that $|x-y|=|x-z|=|y-z|$).
\end{example}
\section{Constructions of $2$-adesigns}
We begin this section with a discussion on the possible number of blocks of $2$-adesigns.
\subsection{Possible number of blocks of $2$-adesigns}\label{secB}
Let $(V,\mathcal{B})$ be a $2$-$(v,k,\lambda)$ adesign with $b$ blocks. According to Lemma \ref{le30.0}, if $(V,\mathcal{B})$ is a tactical configuration, then $\lambda=\lfloor r^{\mathcal{B}}\frac{k-1}{v-1}\rfloor$ and \[
\lambda\binom{v}{2}/\binom{k}{2}<b<(\lambda+1)\binom{v}{2}/\binom{k}{2}.
\] Let $v,k$ and $\lambda$ be positive integers and let $(V,\mathcal{B})$ be an incidence structure with $|V|=v$ and $|B|=k$ for all $B\in\mathcal{B}$. If each pair of points occurs in at least $\lambda$ blocks, then $(V,\mathcal{B})$ is a $(v,k,\lambda)$-{\it covering}. If each pair of points occurs in at most $\lambda$ blocks then $(V,\mathcal{B})$ is a $(v,k,\lambda)$-{\it packing}. The classical bound for coverings is the Schonheim bound \cite{SCHON}, which states that, if $\mathcal{B}$ is a $\lambda$-covering with $b$ blocks, then \[b\geq C_{\lambda}\text{ where }C_{\lambda}:=\left\lceil\frac{v}{k}\left\lceil\frac{\lambda (v-1)}{k-1}\right\rceil\right\rceil,\] and the classical bound for packings is the Johnson bound \cite{JOHN}, and states that, if $\mathcal{B}$ is a $\lambda$-packing with $b$ blocks, then \[b\leq P_{\lambda}\text{ where }P_{\lambda}:=\left\lfloor\frac{v}{k}\left\lfloor\frac{\lambda (v-1)}{k-1}\right\rfloor\right\rfloor.\]
\par
In \cite{HOR}, Horsely showed that, in certain situations, these bounds could be improved. Denote the incidence matrix of a $(v,k,\lambda)$-covering resp. -packing by $M_{c}$ resp. $M_{p}$. Also note that, if $b$ is the number of blocks in $\mathcal{B}$, then \[
b\geq{\rm rank}(M)\geq{\rm rank}(MM^{T})\]where $M$ is either $M_{c}$ or $M_{p}$.
\begin{lemma}\label{le4} {\rm \cite{HOR}} Let $v,k$ and $\lambda$ be positive integers such that $3\leq k<v$, and let $r$ and $d$ be the integers such that $\lambda(v-1)=r(k-1)-d$ and $0\leq d<k-1$. If $d<r-\lambda$, then \[
{\rm rank}(M_{c}M_{c}^{T})\geq C'_{\lambda}(r,d)\text{ where }C'_{\lambda}(r,d):=\left\lceil\frac{v(r+1)}{k+1}\right\rceil.\]
\end{lemma}
\begin{lemma}\label{le5} {\rm \cite{HOR}} Let $v,k$ and $\lambda$ be positive integers such that $3\leq k<v$, and let $r$ and $d$ be the integers such that $\lambda(v-1)=r(k-1)+d$ and $0\leq d<k-1$. If $d<r-\lambda$, then \[
b\leq P'_{\lambda}(r,d)\text{ where }P'_{\lambda}(r,d):=\left\lfloor\frac{v(r-1)}{k-1}\right\rfloor.\]
\end{lemma}
Again let $(V,\mathcal{B})$ be a $2$-$(v,k,\lambda)$ adesign with $b$ blocks. Clearly $(V,\mathcal{B})$ is a $(v,k,\lambda)$-covering and a $(v,k,\lambda+1)$-packing. Let $r_{1},d_{1}$ and $\lambda$ be defined, respectively, as $r,d$ and $\lambda$ were defined in Lemma \ref{le4}, and let $r_{2},d_{2}$ and $\lambda+1$ be defined, respectively, as $r,d$ and $\lambda$ were defined in Lemma \ref{le5}. Then, summing up the above discussion gives us \begin{equation}\label{eq9}
C^{*}\leq b \leq P^{*},
\end{equation} where $C^{*}$ is equal to $C'_{\lambda}(r_{1},d_{1})$ if $r_{1}$ and $d_{1}$ satisfy the conditions of Lemma \ref{le4}, and equal to $C_{\lambda}$ otherwise, and $P^{*}$ is equal to $P'_{\lambda}(r_{2},d_{2})$ if $r_{2}$ and $d_{2}$ satisfy the conditions of Lemma \ref{le5}, and equal to $P_{\lambda}$ otherwise.
\subsection{Constructions of $2$-adesigns}
This section will also be concerned with constructions via adjacency matrices of strongly regular graphs, and we will assume much of the same notation used in Section \ref{sec3}. We open the discussion with the following simple construction.
\par
Let $A$ be the incidence matrix of a strongly regular graph with parameters $(v,k,\lambda,\mu)$ where $\mu=\lambda+1$ or $\lambda+3$. By Lemma \ref{le00} we have that \begin{equation}\label{eq00C}
(A+I)^{2}=(k+1)I+(\lambda+2)A+\mu(J-I-A).\end{equation} Then, by the regularity of the graphs, (\ref{eqAB}) applies, and $(\mathcal{V}_{A},\mathcal{B}_{I+A})$ is a $2$-$(v,k+1,\lambda')$ adesign where $\lambda'=\lambda+1$ or $\lambda+2$.
\begin{example}\label{ex00E} Strongly regular graphs with parameters of the form $(m^{2},d(m-1),d^{2}-3d+m,d(d-1))$ are called {\it pseudo-Latin square} type \cite{VAN}, and are known to exist whenever $m$ is an odd prime power and $2\leq d\leq m$, via orthogonal arrays (see Theorem 6.39 of \cite{STIN} and Section 2.5.2 of \cite{BEH}). Then, taking as $m$ the odd prime power $q$, and setting $d=\frac{q-1}{2}$, we can construct a strongly regular graph whose complementary graph has parameters $(q^{2},\frac{q^{2}+2q-3}{2},\frac{q^{2}+4q-9}{4},\frac{q^{2}+4q+3}{4})$. If $A$ is the adjacency matrix of such a strongly regular graph, then, by (\ref{eq00C}), we have that $(\mathcal{V}_{A},\mathcal{B}_{I+A})$ is a $2$-$(q^{2},\frac{q^{2}+2q-1}{2},\frac{q^{2}+4q-1}{4})$ adesign.
\end{example}
In the next constructions, we again consider unions of row supports of certain strongly regular graphs.
\par
Suppose that $A$ and $A'$ are the adjacency matrices of strongly regular graphs with parameters $(v,k,\lambda,\mu)$ where $\mu=\lambda+1$ or $\mu=\lambda-1$, and the property that either $A+A'$ is a $0$-$1$ matrix, or $A+A'+I$ is a $1$-$2$ matrix (i.e., all of the off-diagonal entries of $A+A'$ are either $1$ or $2$). By Lemma \ref{le00} we have that \begin{equation}\label{eq001}
\mb A & A' \me \mb A \\ A' \me = A^{2} + (A')^{2} = 2kI+(\lambda-\mu)(A+A')+2\mu(J-I). \end{equation} Thus, it follows from (\ref{eqAB}) by the regularity of the graphs, and the fact that either $A+A'$ is a $0$-$1$ matrix or $A+A'+I$ is a $1$-$2$ matrix, that $(\mathcal{V}_{A},\mathcal{B}_{A}\cup\mathcal{B}_{A'})$ is a $2$-$(v,k,\lambda')$ adesign where \[
\lambda'=\begin{cases} 2\mu\text{ or }2\mu-1,&\text{ if }A+A'\text{ is a } 0\text{-}1\text{ matrix,}\\
                       2\mu+1\text{ or }2\mu-2,&\text{ if }A+A'+I\text{ is a } 1\text{-}2\text{ matrix.}\end{cases}\]
Indeed, if $\mu=\lambda+1$, and $A+A'$ is a $0$-$1$ matrix, then (\ref{eq001}) implies that any pair of distinct points of $\mathcal{V}_{A}$ (since $A$ and $A'$ have the same dimensions, we may assume that $\mathcal{V}_{A}=\mathcal{V}_{A'}$) appears in either $2\mu-1$ or $2\mu$ blocks of $\mathcal{B}_{A}\cup\mathcal{B}_{A'}$. The other cases can be checked in a similar way.
\par
Now suppose that $A$ and $A'$ are the adjacency matrices of strongly regular graphs with parameters $(v,k,\lambda,\mu)$ where $\mu=\lambda+1$ or $\mu=\lambda+3$, and the property that either $A+A'$ is a $0$-$1$ matrix, or $A+A'+I$ is a $1$-$2$ matrix. By Lemma \ref{le00} we have \begin{equation}\label{eq002}
\mb A+I & A'+I \me \mb A+I \\ A'+I \me = (A+I)^{2} + (A'+I)^{2} = 2(k+1)I+(\lambda-\mu+2)(A+A')+2\mu(J-I). \end{equation} Thus, it follows from (\ref{eqAB}) by regularity of the graphs and the fact that either $A+A'$ is a $0$-$1$ matrix or $A+A'+I$ is a $1$-$2$ matrix, that $(\mathcal{V}_{A},\mathcal{B}_{A+I}\cup\mathcal{B}_{A'+I})$ is a $2$-$(v,k,\lambda')$ adesign where \[
\lambda'=\begin{cases} 2\mu\text{ or }2\mu-1,&\text{ if }A+A'\text{ is a }0\text{-}1\text{ matrix,}\\
                       2\mu+1\text{ or }2\mu-2,&\text{ if }A+A'+I\text{ is a }1\text{-}2\text{ matrix.}\end{cases}\]
Indeed, if $\mu=\lambda+3$, and $A+A'+I$ is a $1$-$2$ matrix, then (\ref{eq002}) implies that any pair of distinct points of $\mathcal{V}_{A}$ ($=\mathcal{V}_{A'}$) appears in either $2\mu-2$ or $2\mu-1$ blocks of $\mathcal{B}_{A+I}\cup\mathcal{B}_{A'+I}$. The other cases can be checked in a similar way.
\begin{example}\label{ex01E} Let $q$ be an odd prime power. Let $G=(\mathbb{F}_{q},+)\times(\mathbb{F}_{q},+)$, and define \[
D=\{(a,b)\in G\mid a\text{ and }b\text{ are both squares or both nonsquares}\},\]and \[
\tilde{D}=\{(a,b)\in G\mid \text{one of }a,b\text{ is square and the other is nonsquare}\}.\]
By Lemma \ref{le00A} (see Appendix), both $D$ and $\tilde{D}$ are partial difference sets with parameters \\$(q^{2},\frac{q^{2}-2q+1}{2},\frac{q^{2}-4q+3}{4},\frac{q^{2}-4q+7}{4})$. Thus, the incidence matrices $A$ resp. $A'$ of $(G,Dev(D))$ resp. $(G,Dev(\tilde{D}))$ are the adjacency matrices of strongly regular graphs where $A+A'$ is a $0$-$1$ matrix, and $2J-I-(A+A')$ is a $1$-$2$ matrix. Then, from the above assertions, it follows that \begin{enumerate}[(i)]
\item $(G,Dev(D)\cup Dev(\tilde{D}))$ is a $2$-$(q^{2},\frac{q^{2}-2q+1}{2},\frac{q^{2}-4q+3}{2})$ adesign, and
\item $(G,Dev(\overline{D})\cup Dev(\overline{\tilde{D}}))$ is a $2$-$(q^{2},\frac{q^{2}+2q+1}{2},\frac{q^{2}+4q-1}{2})$ adesign.\end{enumerate} In particular, if $q=5$, so that \[
D=\{(2,3),(1,4),(3,2),(4,1),(3,3),(1,1),(4,4),(2,2)\}\] and  \[
\tilde{D}=\{(4,3),(1,2),(1,3),(3,1),(2,1),(2,4),(3,4),(4,2)\},\]then the incidence structure $(\mathbb{F}_{5}\times\mathbb{F}_{5},Dev(D)\cup Dev(\tilde{D}))$ is a $2$-$(25,8,4)$ adesign with 50 blocks.
\end{example}
The above discussion is summarized in the following.
\begin{theorem}\label{th001} Let $A$ and $A'$ be the adjacency matrices of strongly regular graphs with parameters $(v,k,\lambda,\mu)$ and the property that $A+A'$ is a $0$-$1$ matrix or $A+A'+I$ is a $1$-$2$ matrix. \begin{enumerate}[(i)]
\item If $\mu=\lambda+1$ or $\lambda-1$, then $(\mathcal{V}_{A},\mathcal{B}_{A}\cup\mathcal{B}_{A'})$ is a $2$-$(v,k,\lambda')$ adesign, and
\item if $\mu=\lambda+1$ or $\lambda+3$, then $(\mathcal{V}_{A},\mathcal{B}_{A+I}\cup\mathcal{B}_{A'+I})$ is a $2$-$(v,k+1,\lambda')$ adesign,\end{enumerate} where \[
\lambda'=\begin{cases} 2\mu\text{ or }2\mu-1,&\text{ if }A+A'\text{ is a } 0\text{-}1\text{ matrix,}\\
                       2\mu+1\text{ or }2\mu-2,&\text{ if }A+A'+I\text{ is a } 1\text{-}2\text{ matrix.}\end{cases}\]
\end{theorem}
\begin{remark}\label{re01} The strongly regular graphs of pseudo-Latin square type mentioned in Example \ref{ex00E}, or those constructed by Pasechnik in \cite{PAS}, or either of their complements, have parameters which make them good candidates for satisfying the conditions of either part of Theorem \ref{th001}. However, realizing two distinct graphs of either one of these types whose incidence matrices $A$ and $A'$ are such that $A+A'$ is a $0$-$1$ matrix, or $A+A'+I$ is a $1$-$2$ matrix, seems difficult in general. We leave this as an open problem.
\end{remark}
Next we show how the concepts of derived and residual designs can be applied to certain adesigns.
\begin{theorem}\label{th004} Let $A$ be the adjacency matrix of a Paley type $(v,k,\lambda,\lambda+1)$ strongly regular graph. Fix a row of $A$ and let $R$ denote its support. Define \[\mathcal{B}=\{R\cap S\mid S \ (\ne R)\text{ is the support of a row of }A\},\] and let $\mathcal{B}_{\infty}$ denote the set containing all members of $\mathcal{B}$ of size $\lambda+1$, and all members of $\mathcal{B}$ of size $\lambda$ modified by adjoining the point $\infty$. Then $(R\cup \{\infty\},\mathcal{B}_{\infty})$ is a $2$-$(k+1,\lambda+1,\lambda-1)$ adesign.
\end{theorem}
\begin{proof} If $x,y\in R$ are distinct, then the number of members of $\mathcal{B}$ in which they appear together, which is either $\lambda$ or $\lambda+1$, is also the number of members of $\mathcal{B}_{\infty}$ in which they appear together. We want to count the number of members of $\mathcal{B}_{\infty}$ in which $x$ and $\infty$ appear together. Notice there are $k-1$ members $B_{1},...,B_{k-1}\in\mathcal{B}_{\infty}$ containing $x$. The number of $B_{i}$'s also containing $\infty$ is the number of common neighbors of the two vertices corresponding to $R$ and $x$. Since $x\in R$, these two vertices are adjacent, whence the number of common neighbors is $\lambda$.
\end{proof}
The proof of the following corollary is a simple application of the principle of inclusion and exclusion, and so is omitted.
\begin{corollary} Let $A$ be the adjacency matrix of a Paley type $(v,k,\lambda,\lambda+1)$ strongly regular graph. Fix a row of $A$ and let $R$ denote its support. Let $\mathcal{B}_{\overline{\infty}}$ be the set of complements in $\mathcal{V}_{A}\cup \{\infty\}$ of members of $\mathcal{B}_{\infty}$. Then $(R\cup \{\infty\},\mathcal{B}_{\overline{\infty}})$ is a $2$-$(k+1,\lambda+2,\lambda+1)$ adesign.
\end{corollary}
\begin{example} Let $q\equiv 1 \ ({\rm mod} \ 4)$ be a prime power and let $D\subseteq \mathbb{F}_{q}$ be the quadratic residues. If we take $\mathcal{B}=Dev(D)$ then, by Theorem \ref{th004}, $(D\cup\{\infty\},\mathcal{B}_{\infty})$ is a $2$-$(\frac{q+1}{2},\frac{q-1}{4},\frac{q-9}{4})$ adesign with $q-1$ blocks. Moreover, since $\lfloor\frac{vr}{k}\rfloor = q+1$, where $r=\frac{q-1}{2}$, any such adesign is two blocks short of meeting the Johnson bound for packings (See Section \ref{secB}).
\end{example}

Let $(V,\mathcal{B})$ be an incidence structure. Let $p\in V$ and define $\mathcal{B}_{p}=\{B\setminus\{p\}\mid p\in B,B\in\mathcal{B}\}$. The incidence structure $(V\setminus\{p\},\mathcal{B}_{p})$ is called the {\it contraction} of $(V,\mathcal{B})$ at the point $p$. We can obtain new symmetric $2$-adesigns by contracting on a point of any one of the $3$-adesigns constructed in Section \ref{sec3}. It is easy to see that contracting at points of a $3$-adesign will result in a $2$-adesign as long as not all three subsets of points occur in the same number of blocks of the contraction.
\begin{remark}\label{re1} Let $(V,\mathcal{B})$ be any one of the $3$-$(v,k,\lambda)$ adesigns constructed in Section \ref{sec3}, and let $p$ be any point of $V$. It is clear from their proofs that contracting at $p$ gives a $2$-$(v,k,\lambda)$ adesign since we can always find a pair $x,y\in V\setminus\{p\}$ such that $x,y$ and $p$ appear together in $\lambda$ blocks of $\mathcal{B}$, and we can also find another pair $x',y'\in V\setminus\{p\}$ such that $x',y'$ and $p$ appear together in $\lambda+1$ blocks of $\mathcal{B}$. Moreover, contracting at a point $p$ of $(V,\mathcal{B})$ gives a $\lambda$-covering which meets the bound given in Lemma \ref{le4}, i.e., it is a {\em minimal} $\lambda$-{\em covering}.
\end{remark}
\begin{example} Let $q$ be an odd prime power and let $D$ and $\tilde{D}$ be defined as in Example \ref{ex01E}. Denote $\mathbb{F}_{q}\times\mathbb{F}_{q}$ by $V$ and $Dev(D\cup (\mathbb{F}_{q}\times\{0\}))\cup Dev(\tilde{D}\cup (\{0\}\times \mathbb{F}_{q}))$ by $\mathcal{B}$. It was shown in \cite{ZHANG} that $D\cup (\mathbb{F}_{q}\times\{0\})$ is a Paley type partial difference set in $V$, whence by Theorem \ref{th5}, the incidence structure $(V,\mathcal{B})$ is a $3$-$(q^{2},\frac{q^{2}+1}{2},\frac{q^{2}-1}{4})$ adesign with $2q$ blocks. If we contract at the point $(0,0)$, then we get the incidence structure $(V\setminus\{(0,0)\},\mathcal{B}_{(0,0)})$, which, by Remark \ref{re1}, is a $2$-$(q^{2}-1,\frac{q^{2}-1}{2},\frac{q^{2}-1}{4})$ adesign with $b=q^{2}+1$ blocks. Now let $r=\frac{q^{2}+1}{2}$ and $d=\frac{q^{2}-5}{4}$. Then with $v=q^{2}-1,k=\frac{q^{2}-1}{2}$ and $\lambda=\frac{q^{2}-1}{4}$, we have that $\lambda(v-1)=r(k-1)-d$ with $0\leq d<r-\lambda$, and $r$ and $d$ satisfy the conditions of Lemma \ref{le4}. Then, since $(q^{2}-1)(q^{2}+3)/(q^{2}+1)>q^{2}$, we have \[
\left\lceil \frac{v(r+1)}{(k+1)}\right\rceil=\left\lceil \frac{(q^{2}-1)(q^{2}+3)}{(q^{2}+1)}\right\rceil=q^{2}+1=b,\]i.e., $b$ meets the bound given in Lemma \ref{le4}, and $(V\setminus\{(0,0)\},\mathcal{B}_{(0,0)})$ is a minimal $\lambda$-covering.
\end{example}
Our next construction is a modification of Bose's Steiner triple systems \cite{BOS2}.
\begin{theorem} Let $n>3$ be an odd integer, let $G=\mathbb{Z}_{n}\times\mathbb{Z}_{3}$, and let ``$<$'' be any total ordering on $\mathbb{Z}_{n}$ (e.g. $0<1<\cdots<n-1$). Define \begin{eqnarray*}
\mathcal{B}
&=&\left\{\{(a,i),(b,i),(\frac{n+1}{2}(a+b),i+1)\}\bigg\vert a,b\in\mathbb{Z}_{n},a<b\right\}\\
& &\cup\left\{\{(a,i),(b-1,i),(\frac{n+1}{2}(a+b),i+1)\}\bigg\vert a,b\in\mathbb{Z}_{n},a\not<b,a\neq b-1\right\}\\
& &\cup\bigg\{\{(a,0),(a,1),(a,2)\}\bigg\vert a\in\mathbb{Z}_{n}\bigg\}.\end{eqnarray*}Then $(G,\mathcal{B})$ is a $2$-$(3n,3,1)$ adesign (with $3n^{2}-2n$ blocks).
\end{theorem}
\begin{proof} Let $(\alpha,j),(\beta,k)\in G$. It is clear that each block is incident with three points. If $\alpha=\beta$ then the pair occurs in the block $\{(\alpha,0),(\alpha,1),(\alpha,2)\}$ or in the block $\{(\alpha,j),(\alpha-1,j),(\alpha,j+1)\}$ ($=\{(\alpha,j),(\beta-1,j),((n+1)/2\cdot(\alpha+\beta),j+1)\}$) and in no other block.
\par
Now assume $\alpha\neq \beta$. Without loss of generality, we can assume that $\alpha<\beta$. There are three cases depending on the residues $k$ and $j$ modulo $3$:
\newline
{\bf (i)} If $k=j$, the pair $(\alpha,j),(\beta,k)$ occurs in the block $\{(\alpha,k),(\beta,k),((n+1)/2\cdot(\alpha+\beta),k+1)\}$ and in no other block.\newline
{\bf (ii)} If $k=j+1 \ ({\rm mod} \ 3)$, the equation $\frac{n+1}{2}(x+\alpha)=\beta$ has a unique solution $x=\gamma$. Since $\frac{a+a}{2}=a$ for all $a\in\mathbb{Z}_{n}$, i.e., the binary operation $f(a,b):=\frac{n+1}{2}$ is idempotent, and $\alpha\neq\beta$, we have $\alpha\neq\gamma$. If $\gamma<\alpha$, the pair $(\alpha,j),(\beta,k)$ occurs in the block $\{(\alpha,j),(\gamma,j),(\beta,k)\}$ ($=\{(\alpha,j),(\gamma,j),((n+1)/2\cdot(\alpha+\gamma),j+1)\}$), as well as in the block $\{(\alpha,j),(\gamma-1,j),(\beta,k)\}$ ($=\{(\alpha,j),(\gamma-1,j),((n+1)/2\cdot(\alpha+\gamma),j+1)\}$), and in no other block. If $\alpha<\gamma$ the pair occurs in in the block $\{(\alpha,j),(\gamma,j),(\beta,k)\}$ and in no other block.\newline
{\bf (iii)} If $j=k+1 \ ({\rm mod} \ 3)$ then the equation $\frac{n+1}{2}(x+\beta)=\alpha$ has a unique solution $x=\gamma$. Since the binary operation $f(a,b):=\frac{n+1}{2}$ is idempotent, and $\alpha\neq\beta$, we have $\gamma\neq\beta$. If $\gamma<\beta$ then the pair $(\alpha,j),(\beta,k)$ occurs in the block $\{(\gamma,k),(\beta,k),(\alpha,j)\}$ ($=\{(\gamma,k),(\beta,k),((n+1)/2\cdot(\gamma+\beta),k+1)\}$), as well as in the block $\{(\beta,k),(\gamma-1,k),(\alpha,j)\}$ ($=\{(\beta,k),(\gamma-1,k),((n+1)/2\cdot(\beta+\gamma),k+1)\}$), and in no other block. If $\gamma>\beta$ then the pair occurs in the block $\{(\beta,k),(\gamma,k),(\alpha,j)\}$ and in no other block.\newline
This completes the proof.
\end{proof}
\section{Concluding Remarks}
In this correspondence we investigated $t$-adesigns and their links with other combinatorial objects such as balanced incomplete block designs, coverings and packings. We first considered the question of when a $(t+1)$-adesign is a $t$-design or a $t$-adesign, and then we constructed several classes of $3$-adesigns which are, in fact, balanced incomplete block designs. We have also discussed some of the restrictions on the possible sets of feasible parameters for both $2$-adesigns and $3$-adesigns. The $2$-adesigns we constructed have new parameters, and some of them have the interesting property that they are minimal $\lambda$-coverings. We leave the reader with the following open problems: (1) We have yet to find an example of a $t$-adesign which is not a $(t-1)$-design. Must a $t$-adesign always be a $(t-1)$-design? (2) In Section \ref{sec2} we made an initial investigation into when a $t$-adesign is either a $(t-1)$-adesign or a $(t-1)$-design. Is it possible to formulate necessary and sufficient conditions for a $t$-adesign to be a $(t-1)$-design? (3) Do the strongly regular graphs mentioned in Remark \ref{re01} (or any other strongly regular graphs not necessarily coming from partial difference sets) satisfy the conditions of either part of Theorem \ref{th001}?
\section*{Acknowledgment}
The authors are very grateful to the three anonymous referees and to the Coordinating Editor for all of their detailed comments that greatly improved the quality and the presentation of this paper.
\section*{Appendix}
We will need some facts about cyclotomic classes and cyclotomic numbers. Let $q=ef+1$ be a prime power, and $\gamma$ a primitive element of the finite field $\mathbb{F}_{q}$ with $q$ elements. The {\it cyclotomic classes} of order $e$ are given by $D_{i}^{(e,q)}=\gamma^{i}\langle \gamma^{e} \rangle$ for $i=0,1,...,e-1$. The {\it cyclotomic numbers of order $e$} are given by $(i,j)_{e}=|D_{i}^{(e,q)}\cap (D_{j}^{(e,q)}+1)|$. It is obvious that there are at most $e^{2}$ different cyclotomic numbers of order $e$. When it is clear from the context, we will simply denote $(i,j)_{e}$ by $(i,j)$.
\par
We will need to use the cyclotomic numbers of order $2$.
\begin{lemma}\label{le2} {\rm \cite{STO}} For a prime power $q$, if $q \equiv 1$ (mod 4), then the cyclotomic numbers of order two are given by
\begin{eqnarray*}
(0,0) & = & \frac{q-5}{4},                         \\
(0,1) & = & (1,0) = (1,1) = \frac{q-1}{4}.
\end{eqnarray*} If $q \equiv 3$ (mod 4) then the cyclotomic numbers of order two are given by
\begin{eqnarray*}
(0,1) & = & \frac{q+1}{4},                         \\
(0,0) & = & (1,0) = (1,1) = \frac{q-3}{4}.
\end{eqnarray*}
\end{lemma}
\begin{lemma}\label{le00A} Let $q$ be an odd prime power. Let $G=(\mathbb{F}_{q},+)\times(\mathbb{F}_{q},+)$, and define \[
D=\{(a,b)\in G\mid a\text{ and }b\text{ are both squares or both nonsquares}\},\]and \[
\tilde{D}=\{(a,b)\in G\mid \text{one of }a,b\text{ is square and the other is nonsquare}\}.\]Then both $D$ and $\tilde{D}$ are $(q^{2},\frac{q^{2}-2q+1}{2},\frac{q^{2}-4q+7}{2},\frac{q^{2}-4q+3}{2})$ partial difference sets.
\end{lemma}
\begin{proof} The case for $D$ was shown in \cite{ZHANG}. To show that for $\tilde{D}$, we count the number of solutions to the equation \begin{equation}\label{eq5}(a,b)=(a_{1},b_{1})-(a_{2},b_{2}),\end{equation} where $(a_{1},b_{1}),(a_{2},b_{2})\in \tilde{D}$. We use a method similar to that used in \cite{ZHANG}.
\newline
 Assume that $a$ and $b$ are both square. If $a_{1}$ and $a_{2}$ are square and $b_{1}$ and $b_{2}$ are nonsquare, then, using Lemma \ref{le2}, the number of solutions to (\ref{eq5}) is $(0,0)_{2}(1,1)_{2}$. There are three other cases depending on which of $a_{1},a_{2},b_{1}$ and $b_{2}$ are square and which are nonsquare, and the number of solutions to (\ref{eq5}), as we run over these other possibilities, is one of $(0,1)_{2}(1,0)_{2},(1,0)_{2}(0,1)_{2}$ or $(1,1)_{2}(0,0)_{2}$. Summing over all four possibilities, the total number of solutions to (\ref{eq5}) when $a$ and $b$ are both square is $\frac{q^{2}-4q+3}{4}$ (regardless of the residue of $q$ modulo $4$).
 \par
 The other three cases where neither $a$ nor $b$ are zero can be argued similarly. When $a$ and $b$ are both nonsquare, the total number of solutions to (\ref{eq5}) is $\frac{q^{2}-4q+3}{4}$, and when one of $a$ and $b$ is square and the other is nonsquare, the total number of solutions is $\frac{q^{2}-4q+7}{4}$. If $a\neq0$ and $b=0$ then (\ref{eq5}) becomes $(a_{2}a^{-1},b_{2})+(1,0)=(a_{1}a^{-1},b_{1})$ which, using Lemma \ref{le2} again, has $((0,0)_{2}+(1,1)_{2})\frac{q-1}{2}=\frac{q^{2}-4q+3}{4}$ solutions. A similar argument shows that when $a=0$ and $b\neq0$ the number of solutions to (\ref{eq5}) is again $\frac{q^{2}-4q+3}{4}$. Thus, if $x,y\in G$ are distinct, we have that each member of $\tilde{D}$ appears as a difference of two distinct members of $\tilde{D}$, $\frac{q^{2}-4q+7}{4}$ times, and each member of $G\setminus(\tilde{D}\cup\{(0,0)\})$ appears as a difference of two distinct members of $\tilde{D}$, $\frac{q^{2}-4q+3}{4}$ times. This completes the proof.
\end{proof}


\end{document}